\documentclass{amsart}
\usepackage{url}
\usepackage{lineno}

\newtheorem{thm}{Theorem}
\newtheorem{cor}[thm]{Corollary}
\newtheorem{lem}[thm]{Lemma}
\newtheorem{prop}[thm]{Proposition}
\newtheorem{conj}{Conjecture}
\theoremstyle{definition}

\theoremstyle{remark}

\numberwithin{equation}{section}
\newtheorem{exam}{Example}

\newcommand{\subs}{\subseteq}
\newcommand{\alphabet}{\frac{\alpha}{\beta}}

\newcommand{\Kot}{Koteljanskii}

\newcommand{\cK}{\mathcal K}
\newcommand{\cA}{\mathcal A}
\newcommand{\cB}{\mathcal B}
\newcommand{\cD}{\mathcal D}
\newcommand{\cE}{\mathcal E}
\newcommand{\cH}{\mathcal H}

\newcommand{\nt}{\operatorname{nul}}
\newcommand{\rt}{\rho}
\newcommand{\asn}{\operatorname{asn}}

\newcommand{\RR}{\mathbb R}
\newcommand{\RRR}{\RR^{(2^N)}}

\newcommand{\allones}{e}
\newcommand{\expon}{\eta}

\newcommand{\set}[1]{\left\{#1\right\}}
\newcommand{\eps}{\varepsilon}

\newcommand{\lmin}{\lambda_{\rm min}}
\newcommand{\lmax}{\lambda_{\rm max}}

\begin{document}

\title[Bounded Ratios over Positive Definite Matrices]{Bounded Ratios of Products
of Principal Minors of Positive Definite Matrices}
\author{H. Tracy Hall}
\address{Department of Mathematics \\
Brigham Young University\\
Provo, UT  84602}
\email{h.tracy@gmail.com}

\author{Charles R. Johnson}
\address{Department of Mathematics \\
 College of William and Mary \\
 Williamsburg, VA  23185}
\email{crjohnso@math.wm.edu}


\begin{abstract}
Considered is the multiplicative semigroup of ratios of products of principal
minors bounded over all positive definite matrices.
A long history of literature identifies various elements
of this semigroup, all of which lie in a sub-semigroup
generated by Hadamard-Fischer inequalities.  Via cone-theoretic
techniques and the patterns of nullity among positive
semidefinite matrices, a semigroup containing all bounded
ratios is given.  This allows the complete determination
of the semigroup of bounded ratios for $4$-by-$4$ positive definite
matrices, whose $46$ generators include
ratios not implied by Hadamard-Fischer and ratios
not bounded by $1$.
For $n \ge 5$ it is shown that the containment of semigroups
is strict, but a generalization of nullity patterns,
of which one example is given, is conjectured
to provide a finite determination of all bounded ratios.
\end{abstract}

\maketitle

\section{Introduction}
For an $n$-by-$n$ matrix $A=(a_{ij})$ and an index set $S\subs N \equiv
\set{1,2,\ldots,n}$, $A[S]$ is the principal submatrix of $A$ lying in the rows
and columns indicated by $S$.
Given a collection $\alpha: \alpha_1, \alpha_2, \ldots, \alpha_p \subs N$ of index
sets and exponents $\expon_1, \ldots, \expon_p \in \RR^+$, let
\[ \alpha(A) = \prod_{i=1}^p (\det A[\alpha_i])^{\expon_i}. \]
At times we also abbreviate $\alpha(A)$ as
$\alpha_1^{\expon_1}\cdots\alpha_p^{\expon_p}$,
omitting any exponents equal to $1$.
We are interested in understanding which ratios
\[
\frac{\alpha}{\beta} = \frac{\alpha(A)}{\beta(A)}
\]
have an upper bound that is independent of $A$,
provided that $A$ is positive definite (PD).
We call such ratios {\em bounded}, and
{\em absolutely bounded} if $1$ is an upper bound.
The Hadamard-Fischer inequality \cite{hj} provides
a classical family of absolutely bounded ratios
in which
$\alpha_1 = S \cup T$ and $\alpha_2 = S \cap T$,
while $\beta_1 = S$ and $\beta_2 = T$, for any two
index sets $S,T \subs N$, with all exponents $1$.
This classical family is in fact bounded over
any class of invertible matrices with positive
principal minors and weakly sign-symmetric
non-principal minors \cite{c}, and we
will call such ratios \Kot\ ratios \cite{k}.
Historical information on determinantal inequalities for
PD matrices may be found in (\cite{hj}, ch. 7).
Over real symmetric matrices, bounded ratios
are equivalent to multiplicative inequalities
on the face volumes of a parallelotope.

A ratio is determined, up to equivalence, by the
overall exponent (numerator minus denominator) of
each nonempty subset of $N.$
The exponent of the empty set is chosen
to make the sum of exponents $0$,
and the collection of $2^n$ exponents is called
the {\em formal logarithm} of $\frac{\alpha}{\beta}$,
written $\log (\frac{\alpha}{\beta})$.
Formal logarithms span a hyperplane in $\RRR$
orthogonal to the all-$1$'s vector $\allones$.

Let $\cB_n$ be the set of bounded ratios,
$\cA_n$ the set of absolutely bounded ratios,
and $\cK_n$ the set of products of positive powers of \Kot\ ratios.
Then, for all $n$, $\cK_n \subs \cA_n \subs \cB_n$.
These are multiplicative semigroups,
and there are corresponding cones
$\log (\cK_n) \subs \log(\cA_n) \subs \log (\cB_n)$.
All previously known multiplicative determinantal inequalities lie in $\cK_n$.
In \cite{jb}, necessary conditions for boundedness
were given, implicitly defining a semigroup $\cE_n$
with $\cB_n \subs \cE_n$, and it was shown that
$\cK_3 = \cE_3$ but $\cA_4 \ne \cE_4$.
Here, we define a new semigroup $\cD_n$ by imposing
linear inequalities on $\log (\cD_n)$ coming from
rank deficient positive semidefinite (PSD) matrices.
We show that $\cB_n \subs \cD_n \subs \cE_n$
for all~$n$.
We also show that $\cA_4 \ne \cB_4 = \cD_4 \ne \cE_4$,
and give explicit generators for $\cB_4$.
Going further, we show that $\cB_5 \ne \cD_5$, but we conjecture
that for each $n$ there exists some finite system
of linear inequalities,
generalizing those coming from rank deficient PSD matrices,
that is sufficient to characterize $\log (\cB_n)$.

\section{Cone theoretic determination of $\cD_n$}
We call a ratio $\frac{\alpha}{\beta}$ {\em homogeneous} if
$\alpha(D)/\beta(D) = 1$ for every PD diagonal matrix $D$,
and call the set of all such ratios $\cH_n$.
All bounded ratios are homogeneous \cite{jb},
and $\log(\cH_n)$ is the subspace of $\RRR$ that is
orthogonal to the following set of $n + 1$ vectors:
the all-$1$'s vector $\allones$ and, for each
$i \in N$, the
vector whose entry $T$ (as $T$ ranges over $T \subs N$)
is $1$ whenever $i \in T$, and $0$ otherwise.
A homogeneous ratio is uniquely determined by
the exponents of subsets with cardinality
at least $2$.
Since the dimension of $\log(\cK_n)$ is
$2^n - n - 1$, the cone $\log (\cB_n)$ has full dimension
within the subspace $\log(\cH_n)$.

Let $M$ be a matrix with $n$ columns.
We define $\nt(M) \in \RRR$, the
{\em nullity type} of $M$,
as the vector of null space dimensions
for subsets of the columns of $M$.
(By convention the empty set of columns has
nullity $0$.)
The nullity type of $M$ also gives
the multiplicity of $0$ as
an eigenvalue for any principal
submatrix of $M^*M$.
For any $n$ there exist finitely
many distinct nullity types
(corresponding to representations of labeled matroids on $n$ elements).

Our first result is that the inner product of the formal logarithm
of a ratio with any nullity type gives a necessary condition
for boundedness.

\begin{thm}\label{thm:nec}
Let $\alphabet$ be a homogeneous ratio, let $M$
be a matrix with $n$
columns,
and define a family of matrices $A_\eps$, $0 < \eps < 1$,
by
\[
  A_\eps = M^*M + \eps I.
\]
Then $\alphabet$ is bounded over the family $A_\eps$
if and only if
$\log(\alphabet)^T \nt(M) \ge 0$.
\end{thm}
\begin{proof}
Let $s=\log(\alphabet)^T \nt(M)$,
and let $r$ represent the sum of the
exponents in $\alpha$,
which is also the sum of the exponents in $\beta$.
Let $\Lambda$ be the set of all nonzero eigenvalues of $M^*M$ and
its principal submatrices,
and let $\lmin$ and $\lmax$ denote
respectively the minimum and maximum elements
of $\Lambda \cup \{1\}$,
so that $\lmin \le 1 \le \lmax$.
For any $0<\eps<1$, the matrix $A_\eps$ is PD,
and the quantity $\alphabet(A_\eps)$ can be expressed
entirely in terms of the eigenvalues of principal submatrices of $A_\eps$.
Each of these eigenvalues is $\eps$ or lies
within the range $\lmin+\eps$ to $\lmax+\eps$.
The total exponent of $\eps$ in the eigenvalue product
is exactly the inner product
$\log(\alphabet)^T \nt(M)=s$.
The contribution of all remaining eigenvalues
either to the numerator $\alpha(A_\eps)$
or to the denominator $\beta(A_\eps)$
lies within the range $\lmin^{nr}$
to $(\lmax+1)^{nr}$.
Thus we have
\[
\eps^s\left(\frac{\lmin}{\lmax+1}\right)^{nr}
\le
\alphabet(A_\eps)
\le
\eps^s\left(\frac{\lmax+1}{\lmin}\right)^{nr}.
\]
If $s < 0$ the left inequality implies that
$\alphabet(A_\eps)$ is unbounded,
and if $ s \ge 0$ the right inequality
implies that $\alphabet(A_\eps)$ is bounded by a constant.
\end{proof}

Given $n$, let
$\nu_1, \ldots, \nu_{\ell}$ be a complete list
of the nullity types of matrices
with $n$ columns.
The {\em dual nullity} semigroup $\cD_n$ is defined as the
set of homogeneous ratios $\alphabet$
such that $\log(\alphabet)^T \nu_i\ge 0$
for all $1 \le i \le \ell$.

\begin{cor}
  For all $n$, $\cB_n \subs \cD_n$.  Moreover, if each extreme ray
  of $\log(\cD_n)$ represents a bounded ratio,
  then $\cB_n = \cD_n$.
\end{cor}

The dual nullity cone $\log(\cD_n)$ is a finite intersection
of half-spaces
and thus its set of extreme rays can
(in principle, and sometimes in practice)
be explicitly calculated.
Some nullity types are redundant to this calculation:
$\nt(I)$ is trivial;
$\nt([ M_1 | \mathbf 0])
 -
 \nt(M_1 \oplus I)$
is orthogonal to $\cH_n$,
so one of these may be omitted;
and $\nt(M_1 \oplus M_2)
= \nt(M_1 \oplus I) + \nt(I \oplus M_2)$.
Given a matrix $M$ with $n$ columns,
let $\rt(M)$ be the vector in $\RRR$
whose entry for each $T \subs N$
is the rank of the subset $T$ of the
columns of $M$.
Since $\nt(M) + \rt(M)$
is the vector $\left[|T|\right]$ of cardinalities, which
is orthogonal to $\cH_n$,
$\nt(M)$ and $-\rt(M)$ are interchangeable
for the purposes of calculating $\cD_n$.
This fact can be useful computationally
when $M$ has low rank.

\begin{exam}
\label{ex1}
Let $n=3$, and let $\nu_1, \ldots, \nu_5$ be the
nullity types of the matrices
\[
  \left[
    \begin{array}{ccc}
      1 & 1 & 1
    \end{array}
  \right],
  \left[
    \begin{array}{ccc}
      0 & 1 & 1
    \end{array}
  \right],
  \left[
    \begin{array}{ccc}
      1 & 0 & 1
    \end{array}
  \right],
  \left[
    \begin{array}{ccc}
      1 & 1 & 0
    \end{array}
  \right],
  \mbox{ and }
  \left[
    \begin{array}{ccc}
      1 & 0 & 1 \\
      0 & 1 & 1
    \end{array}
  \right].
\]
Let $\cE_3$ be the intersection of the five sets
$ \{\alphabet \in \cH_3 :\ \log(\alphabet)^T \nu_i \ge 0\}.$
The six extreme rays of $\log(\cE_3)$ are the formal logarithms of
the \Kot\ ratios
\[
  \frac{\{12\}\,\emptyset\ }{\{1\}\{2\}},
  \frac{\{13\}\,\emptyset\ }{\{1\}\{3\}},
  \frac{\{23\}\,\emptyset\ }{\{2\}\{3\}},
  \frac{\{123\}\{1\}}{\{12\}\{13\}},
  \frac{\{123\}\{2\}}{\{12\}\{23\}},
  \mbox{ and }
  \frac{\{123\}\{3\}}{\{13\}\{23\}}.
\]
We thus have $\cK_3 \subs \cA_3 \subs \cB_3 \subs \cD_3 \subs \cE_3 \subs \cK_3$
and in particular $\cB_3 = \cK_3$, as was shown in \cite{jb}.
\end{exam}

The symmetric group on $n$ elements
acts naturally on $\RRR$ and
preserves
every set we have defined.
Set complementation $S^c = N \setminus S$
also acts naturally and preserves $\log(\cK_n)$ and $\log(\cH_n)$.
Since
$\frac{S}{\emptyset}(A) = \frac{S^c}{N}(A^{-1})$
by Jacobi's identity \cite{hj},
$\log(\cA_n)$ and $\log(\cB_n)$ are also invariant under
the action of complementation.
To show that $\log(\cD_n)$ is invariant
under complementation requires
some matroid theory~\cite{oxley}:
the complement of a nullity type is
equivalent, up to vectors orthogonal
to $\log(\cH_n)$,
to the nullity type of the dual matroid.

Theorem~\ref{thm:nec} generalizes
the known set-theoretic necessary conditions ST0, ST1, and ST2.
\begin{thm}
\cite{jb} For any bounded ratio $\alphabet$ and any given
index set $S \subs N$, the sum of exponents in $\alphabet$ for
supersets of $S$ is nonnegative (ST1), and likewise for subsets of $S$ (ST2).
Any ratio satisfying ST1 and ST2
also satisfies homogeneity (ST0).
\end{thm}
We define $\cE_n$ as the set of ratios satisfying
the elementary necessary conditions
ST1 and ST2 for every $S \subs N$.
These conditions can be restated in terms of nullity types.
Given an index set $S \subs N$ with
$|S| \ge 3$, define
$M_{S}$ to be the matrix of size $(n-1) \times n$
that is a permutation of $[I | e] \oplus I$,
with the $[I | e]$ summand occurring in the
columns of $S$ and the $I$ summand occurring
in the complement of $S$.
Then for $T \subs N$,
entry $T$ of $\nt(M_{S})$ is $1$
if $S \subs T$ and $0$ otherwise.
Given an index set $S \subs N$ with
$|S| \le n-2$, define
$M^{S}$ to be the matrix of size $1 \times n$
whose entries are $0$ in the columns
of $S$ and $1$ elsewhere.
Then for $T \subs N$,
entry $T$ of $\rt(M^S)$ is $0$
if $T \subs S$ and $1$ otherwise.
For example, with $n=3$, the five matrices listed in Example~\ref{ex1}
are $M^{\emptyset}$,
$M^{\{1\}}$,
$M^{\{2\}}$,
$M^{\{3\}}$,
and $M_{\{123\}}$.
\begin{prop}
\label{prop:one}
A homogeneous ratio $\alphabet$ satisfies
conditions ST1 and ST2 if and only if
$\log(\alphabet)^T \!\nt(M_S) \ge 0$
for $|S| \ge 3$ and
$\log(\alphabet)^T \!\nt(M^S) \ge 0$
for $|S| \le n-2$.
\end{prop}
\begin{proof}
The cases $|S| \le 1$ for ST1 and
$|S| \ge n-1$ for ST2 are equivalent
to the assumed homogeneity.
The case ST1 with $|S|=2$ will follow
from ST2 with $|S| = n-2$:
for a homogeneous ratio,
ST1 is satisfied for $\{i,j\}$ if and only if
ST2 is satisfied for $\{i,j\}^c$,
as can be seen
by partitioning $2^N$ into four groups
depending on the membership of $i$
and $j$, and comparing
the total exponents of these groups under homogeneity.
For $|S| \ge 3$,
$\log(\alphabet)^T \nt(M_S)$
is nonnegative if and only if
$S$ occurs as a subset
in $\alpha$ with a total exponent
at least as great as in $\beta$.
For $|S| \le n-2$,
$\log(\alphabet)^T \left(\allones - \rt(M^S)\right)$
is nonnegative if and only if
$S$ occurs as a superset
in $\alpha$ with a total exponent
at least as great as in $\beta$.
Since $\alphabet$ is homogeneous
and $\nt(M^S) + \rt(M^S) - \allones$
is orthogonal to $\log(\cH_n)$,
this is equivalent to
$\log(\alphabet)^T \nt(M^S) \ge 0$.
\end{proof}

\begin{cor}
For all $n$, $\cD_n \subs \cE_n$.
\end{cor}

The name $\cE_3$ in Example~\ref{ex1} is
now justified, and we have shown,
as in~\cite{jb}, that $\cK_3 = \cE_3$.
Since the vectors $\nt(M^S)$ for $|S| \le n-2$
and the vectors orthogonal
to $\log(\cH_3)$ together span $\RRR$,
$\log(\cB_n)$ is a pointed cone:
no nontrivial ratio is both
bounded and bounded away from $0$.

\section{Bounding ratios not in $\cK_n$}
Given a ratio $\alphabet$
in $\cD_n \setminus \cK_n$,
it may be difficult to determine whether
$\alphabet$ is bounded.
We introduce a technique that can sometimes prove boundedness, starting
with a known observation
whose first nontrivial case is $n = 3$:

\begin{lem}
\label{lem:Fied}
\cite{fiedler}
Let $A=(a_{ij})$ be an $n$-by-$n$ PD matrix with inverse
$B=(b_{ij})$.  Then
for each $i=1,\ldots,n$
\begin{equation}\label{Fied}
2\sqrt{a_{ii}b_{ii}}+(n-2)\leq\sum_{j=1}^n\sqrt{a_{jj}b_{jj}}.
\end{equation}
\end{lem}

Eliminating a copy of $i$ on each side,
replacing the sum with $(n-1)$ times the maximum, and using the fact that
$b_{ii} = \frac{\{i\}^c}{N}(A)$, we conclude:

\begin{cor}\label{cor:Fied}
For $n \ge 3$ and $A$ PD, and for any $i \in N$,
\[ \min_{j \ne i}\set{\frac{\set{i}\set{i}^c}{\{j\}\{j\}^c}(A)} \leq (n-1)^2. \]
\end{cor}

\begin{exam}
\label{ex2}
In~\cite{jb} it is shown that $\cA_4 \ne \cE_4$ using a permutation of
\[
  R_1
  =\frac{\set{124}\set{134}\set{23}\set{1}\set{4}}
          {\set{12}\set{13}\set{14}\set{24}\set{34}}. \]
Observe that this ratio is
unchanged when indices $2$ and $3$ are interchanged;
we may thus assume without loss of generality that
$\frac{\set{1}\set{23}}{\set{2}\set{13}}(A) \le \frac{\set{1}\set{23}}{\set{3}\set{12}}(A)$.
By applying Corollary~\ref{cor:Fied}
to $A[\set{123}]$ with $i=1$ we then obtain,
factoring out a pair of \Kot\ ratios,
\[
  R_1 =
    \frac{\set{124}\set{2}}{\set{12}\set{24}} \cdot
    \frac{\set{134}\set{4}}{\set{14}\set{34}} \cdot
    \frac{\set{1}\set{23}}{\set{2}\set{13}}
  \le 4,
\]
which shows that $R_1$ belongs to $\cB_4$.
\end{exam}

\section{The extreme rays of $\log(\cB_4)$}
Let $\mu_1, \ldots, \mu_{23}$ represent the distinct nullity types
of column permutations of the following seven matrices:
$M^{\emptyset}$, $M^{\{1\}}$,
$M^{\{12\}}$,
$M_{\{123\}}$,
$M_{\{1234\}}$,
\[
  M_6 =
  \left[
    \begin{array}{cccc}
    1 & 0 & 1 & 1 \\
    0 & 1 & 1 & 1
    \end{array}
  \right],
\mbox{ and }
  M_7 =
  \left[
    \begin{array}{cccc}
     1 & 1 & 1 & 1 \\
     0 & 1 & 2 & 3
    \end{array}
  \right].
\]

The intersection of $\cH_4$ with the $23$ half-spaces
$ \{x \in \RR^{16} :\ x^T\mu_i \ge 0\}$
contains $\cD_4$ and thus $\cB_4$.
There are various software packages available for computing
the extreme rays of an intersection of half-spaces in exact
arithmetic; in this case the calculation was done with
custom-written software, whose results were later verified
with the program {\tt lrs} \cite{lrs}
(available at \url{http://cgm.cs.mcgill.ca/~avis/C/lrs.html}).
The $46$ extreme rays
include the $24$ extreme rays of $\log(\cK_4)$
(those with $|S| = |T| = |S \cap T| + 1$) and
the $6$ permutations
of $\log(R_1)$ from Example~\ref{ex2}.
The ratios corresponding to the remaining $16$ extreme rays are
permutations of
\begin{equation}\label{r2}
R_2 =
\frac
   {\set{1234}\set{1234}\set{23}\set{24}\set{34}
    \set{1}\,\emptyset\ }
   {\set{123}\set{124}\set{134}\set{234}
    \set{2}\set{3}\set{4}},
\end{equation}
\begin{equation}\label{r3}
R_3 =
\frac
   {\set{1234}\set{23}\set{24}\set{34}
    \set{1}\set{1}\,\emptyset\ }
   {\set{123}\set{12}\set{13}\set{14}
    \set{2}\set{3}\set{4}},
\end{equation}
and their complements.  Since $R_2$ and $R_3$
are both symmetric in indices $2$ and $3$,
we may make the same assumption as in
Example~\ref{ex2} and conclude
from Corollary~\ref{cor:Fied} that
\[
      R_2 =
         \frac{\set{1234}\set{24}} {\set{124}\set{234}} \cdot
         \frac{\set{1234}\set{13}} {\set{123}\set{134}} \cdot
         \frac{\set{34}\,\emptyset\ } {\set{3}\set{4}} \cdot
            \frac{\set{1}\set{23}}{\set{2}\set{13}} \le 4
\]
and
\[
    R_3 =
         \frac{\set{1234}\set{24}} {\set{124}\set{234}}\cdot
         \frac{\set{124}\set{1}} {\set{12}\set{14}}\cdot
         \frac{\set{34}\,\emptyset\ } {\set{3}\set{4}}\cdot
         \frac{\set{1}\set{23}}{\set{2}\set{13}} \le 4.
\]
Although the proven bound is $4$,
experimental evidence suggests that the ratios $R_2$ and $R_3$
are absolutely bounded,
which would imply $\cK_4 \ne \cA_4$,
and that the correct upper bound for $R_1$ is $\frac{27}{16}$,
which value it is known to approach~\cite{hm}.
This is the first known case
showing $\cA_n \ne \cB_n$ for any $n$.
It has been shown \cite{fhj} that
bounded ratios
on M-matrices or inverse M-matrices
are absolutely bounded, and
preliminary work shows the
same for (invertible) totally nonnegative matrices as far as $n<6$.

The listed nullity types are thus sufficient to
define $\cD_4$, and indeed $\cB_4$.

\begin{thm}
The semigroup $\cB_4$
is generated, up to permutation and complementation,
by positive powers of $\frac{\set{12}\,\emptyset\ }{\set{1}\set{2}}$,
$\frac{\set{123}\set{1}}{\set{12}\set{13}}$, $R_1$,
$R_2$, and $R_3$.
\end{thm}
\begin{thm}
The sets $\cB_4$ and $\cD_4$ are equal: a ratio $\alphabet \in \cH_4$
is bounded if and only if
$\log(\alphabet)^T\mu_i \ge 0$
for all $i = 1, \ldots, 23$, with $\mu_i$ as listed above.
\end{thm}

We end the section with an example showing that $\cD_4 \ne \cE_4$.
The ratio
\[
\alphabet=
\frac{\set{1234}\set{134}\set{12}\set{14}\set{23}\set{24}\set{3}\,\emptyset\ }
{\set{123}\set{124}\set{234}\set{13}\set{34}\set{1}\set{2}\set{4}}
\]
satisfies ST1 and ST2 with respect to every $S \subs \{1,2,3,4\}$.
However, the inner product $\log(\alphabet)^T \nt(M_6)$ gives $-1$, and
$\alphabet$ is not bounded.

\section{The case $n = 5$}
A ratio in the semigroup $\cD_n$ is bounded
near $\eps = 0$ for any
matrix family of the form
$\left(B+\eps C\right)^*\left(B+\eps C\right)$,
where $B$ is singular
and $C$ is generic in the sense that
$C x = B y$ and $B x =  0$
imply $ x =  0$.
We have shown that for $n=4$ such matrix families,
whose behavior depends only on the nullity type of $B$,
are sufficient to determine $\cB_4$.
The following example shows that for $n \ge 5$,
more general
matrix families must be considered.
\begin{exam}
\label{ex3}
Let $n=5$ and consider the ratio
\[
  Q =
  \frac{\set{12345}\set{1345}\set{2345}\set{123}\set{125}\set{34}\set{45}\set{1}\set{2}\,\emptyset\ }
  {\set{1234}\set{1235}\set{145}\set{245}\set{345}\set{12}\set{13}\set{23}\set{4}\set{5}}.
\]
We claim that $Q \in \cD_5$.
Since $Q$ is self-complementary, it is only necessary to
check against matrices of rank up to $\lfloor \frac52 \rfloor$,
i.e.\ $M$ of size $2 \times 5$.
Uniformly deleting a specified index element in the sets of $Q$ yields a product
of \Kot\ ratios in all five cases, which eliminates the
case in which $M$ has a zero column (using homogeneity).
The nullity type of $M$ now depends only on a partition
of the columns $\{1,2,3,4,5\}$ into parallel classes.
Partitions into exactly two blocks give the nullity
type of a direct sum and are thus redundant,
leaving $37$ cases to check.

Now let
\[
  P =
  \left[
    \begin{array}{ccccc}
      1  & 1  & 1  & 1  & 1   \\
      0  &\eps& 0  & 1  & 2   \\
      0  & 0  &\eps& 1  & 2   \\
      0  & 0  & 0  &\eps& 0   \\
      0  & 0  & 0  & 0  &\eps \\
    \end{array}
  \right],
\]
define $A_\eps$ as $P^*P$,
and consider the polynomial $\det A_\eps[ S ]$ for some $S \subs N$.
This is a nonnegative function of $\eps$, and so for every $S$
the dominating term for $\eps$ small must
be an even power $C_S\eps^{2d_S}$ for $C_S > 0$ and $d_S$ a nonnegative integer.
Call the vector $\left[d_S\right] \in \RRR$
the {\em asymptotic nullity type} of $P$,
or $\asn(P)$.
The fact that $\log(Q)^T \asn(P) = -1$
means that the ratio $Q$ is not bounded, and $\cB_5 \ne \cD_5$.
\end{exam}

We expect that any one-parameter matrix
family making a particular ratio unbounded
can be adequately approximated by
polynomials in $\eps$,
and further that the maximum necessary degree of
the approximation depends only on $n$.
We thus expect $\log(B_n)$ to be
a polyhedral cone for all $n$,
as it is for $n \le 4$.

\begin{conj}
Given any $n$,
define the asymptotic nullity type $\asn(P)$ of an
invertible $n \times n$ matrix $P$ of polynomials
in $\eps$
  as in Example~\ref{ex3}.
Then there exists a finite list of polynomial matrices
$P_1$, \ldots, $P_\ell$
such that a homogeneous ratio $\alphabet$
is bounded
if and only if
$\log(\alphabet)^T \asn(P_i) \ge 0$
for all $1 \le i \le \ell$.
\end{conj}

\bibliographystyle{amsplain}
\bibliography{hj4}

\providecommand{\bysame}{\leavevmode\hbox to3em{\hrulefill}\thinspace}
\providecommand{\MR}{\relax\ifhmode\unskip\space\fi MR }
\providecommand{\MRhref}[2]{%
  \href{http://www.ams.org/mathscinet-getitem?mr=#1}{#2}
}
\providecommand{\href}[2]{#2}
\begin{thebibliography}{1}

\bibitem{lrs}
David Avis, \emph{lrs: A revised implementation of the reverse search vertex
  enumeration algorithm}, Polytopes---combinatorics and computation
  (Oberwolfach, 1997), DMV Sem., vol.~29, Birkh\"auser, Basel, 2000,
  pp.~177--198.

\bibitem{c}
David Carlson, \emph{Weakly sign-symmetric matrices and some determinantal
  inequalities}, Colloq. Math. \textbf{17} (1967), 123--129.

\bibitem{fhj}
Shaun Fallat, H.~Tracy Hall, and Charles~R. Johnson, \emph{Characterization of
  product inequalities for principal minors of m-matrices and inverse
  m-matrices}, Quart. J. Math. Oxford Ser. (2) \textbf{49} (1998), 451--458.

\bibitem{fiedler}
Miroslav Fiedler, \emph{Relations between the diagonal elements of two mutually
  inverse positive definite matrices}, Czechoslovak Math. J. \textbf{14 (89)}
  (1964), 39--51.

\bibitem{hj}
Roger~A. Horn and Charles~R. Johnson, \emph{Matrix analysis}, Cambridge
  University Press, Cambridge, 1985.

\bibitem{jb}
Charles~R. Johnson and Wayne~W. Barrett, \emph{Determinantal inequalities for
  positive definite matrices}, Discrete Math. \textbf{119} (1993), no.~1-3,
  97--106, ARIDAM IV and V (New Brunswick, NJ, 1988/1989).

\bibitem{k}
David~M. Koteljanskii, \emph{The theory of nonnegative and oscillating
  matrices}, Amer. Math. Soc. Transl. (2) \textbf{27} (1963), 1--8.

\bibitem{hm}
Peter McNamara, \emph{Ratios of minors of positive semidefinite matrices},
  Research Experiences for Undergraduates Program, College of William and Mary
  (Summer 1997), Advisor: H. Tracy Hall.

\bibitem{oxley}
James~G. Oxley, \emph{Matroid theory}, Oxford Science Publications, The
  Clarendon Press Oxford University Press, New York, 1992.

\end{thebibliography}

\end{document}